\documentclass[noinfoline
]{imsart}
\usepackage{amsmath,amssymb,amsthm,amsfonts,amsbsy}
\usepackage{graphicx}

\usepackage{enumerate}
\usepackage[final]
{showkeys}

\usepackage[text={5.7in,8.69in},centering]{geometry}

\usepackage{hyperref}

\providecommand{\texorpdfstring}[2]{#1}
\providecommand{\url}[1]{#1}

\renewcommand{\le}{\leqslant}
\renewcommand{\ge}{\geqslant}

\renewcommand{\O}{\underset{\ffrown}{<}}

\renewcommand{\P}{\operatorname{\mathsf{P}}} 
\newcommand{\E}{\operatorname{\mathsf{E}}} 
\newcommand{\Var}{\operatorname{\mathsf{{V}ar}}} 
  
\newcommand{\ii}[1]{\operatorname{\mathsf{I}}\{#1\}}

\DeclareMathOperator*{\avee}{\operatorname{\mathsf{ave}}}

\newcommand{\are}{\mathsf{ARE}\,}  

\newcommand{\AS}{\mathsf{AS}}
\newcommand{\AL}{\mathsf{AL}}                
\newcommand{\OS}{\mathsf{OS}}
\newcommand{\OL}{\mathsf{OL}} 

\renewcommand{\AA}{\mathsf{A}}
\renewcommand{\O}{\mathsf{O}}
\newcommand{\SSS}{\mathsf{S}}
\renewcommand{\L}{\mathsf{L}}

\renewcommand{\mod}{\mathrm{mod}}

\newcommand{\R}{\mathbb{R}}                  
\newcommand{\N}{\mathbb{N}}

\newcommand{\al}{\alpha}
\newcommand{\be}{\beta}

\newcommand{\de}{\delta}
\newcommand{\De}{\Delta}

\renewcommand{\th}{\theta}
\newcommand{\Th}{\Theta}

\newcommand{\tth}{\vartheta}

\newcommand{\si}{\sigma}

\newcommand{\tS}{\tilde S}

\newcommand{\vpi}{\varphi}

\newcommand{\Om}{\Omega}

\newcommand{\tu}{{\tilde u}}
\newcommand{\tv}{{\tilde v}}

\newcommand{\XX}
{\R^d}

\newcommand{\A}{\mathcal{A}}

\newcommand{\U}{\mathcal{U}}

\newcommand{\X}{\mathbf{X}}
\newcommand{\Y}{\mathbf{Y}}

\newcommand{\dd}{\operatorname{d}\!}
\newcommand{\sign}{\operatorname{sign}}

\newcommand{\e}{E}

\newtheorem*{theorem*}{Theorem (Chen--Shao \cite{chen07})}
\newtheorem{theorem}{Theorem}
\newtheorem{corollary}[theorem]{Corollary}
\newtheorem{lemma}[theorem]{Lemma}
\newtheorem{proposition}[theorem]{Proposition}
\theoremstyle{remark}
\newtheorem{remark}[theorem]{Remark}

\numberwithin{equation}{section}
\numberwithin{theorem}{section}

\begin{document} 

\begin{frontmatter}

\title{Asymptotic relative efficiency of the Kendall and Spearman correlation statistics}
\runtitle{
ARE, Kendall's and Spearman's}

\begin{aug}
 \author{\fnms{Iosif}  \snm{Pinelis}\corref{}\ead[label=e1]{ipinelis@mtu.edu}}
 \runauthor{I. Pinelis}
 \affiliation{Michigan Technological University}
 \address{Department of Mathematical Sciences\\Michigan Technological University\\Hough\-ton, Michigan 49931\\\printead{e1}}
\end{aug} 

\begin{abstract} 
A necessary and suffcient condition for Pitman's asymptotic
relative effciency (ARE) of the Kendall and Spearman correlation statistics for the independence test to be 1 is given, in terms of certain smoothness and nondegeneracy properties of the model. Corresponding easy to use and  broadly applicable sufficient conditions are obtained, which are then illustrated on several known models of dependence. Effects of the presence or absence of the smoothness and/or nondegeneracy parts of the mentioned necessary and suffcient condition are demonstrated using certain specially
constructed dependence models. A more general (than usual) version of Pitman's ARE is developed, with broader and easier to check conditions of applicability. This version of the ARE, which is then used in the rest of the paper, may also be of value elsewhere. 
%
\end{abstract}



%


\end{frontmatter}

\tableofcontents 

\section{Introduction}
\label{intro}

The correlation statistics most commonly used in nonparametric tests of independence are Kendall's $T$ and Spearman's $S$. How do they compare with each other in terms of efficiency?


Quoting from \cite[page~844]{kruskal58}:

\begin{quote}
As between $r_S[=S]$ and $t[=T]$ qua estimators, they are not really in general competition,
for they are estimators of different population quantities. It may be noted
that $\tau[=\E T]$ is simpler to interpret than $\rho_S[=\frac{n+1}{n-2}\,\E S-\frac3{n-2}\,\E T\approx\E S\text{ for large }n]$. Some authors consider it important that
$r_S$ is easier to compute than $t$. \\ 
{[...]} \\ 
On the grounds of simplicity of interpretation, reasonable sensitivity to
form of distribution, and relative simplicity of sampling theory, I prefer the
use of $\tau$ and $t$ to that of $\rho_S$ and $r_S$.
\end{quote}

Kendall and Gibbons (1990) \cite{ken48} argue that confidence
intervals for Spearman's $S$ are less reliable and less interpretable than confidence intervals
for Kendall's $T$.

It is concluded in \cite{croux-dehon} that for the normal model ``The Kendall correlation measure is more robust
and slightly more efficient than Spearman's rank correlation''. 

Even though, as mentioned in the above quote from \cite{kruskal58}, $S$ and $T$ may not be in general competition as estimators, they can be compared as test statistics, especially in independence tests (but also, more generally, in testing the strength of an association). This comparison will be made in the present paper in terms of Pitman's asymptotic relative efficiency ($\are$), introduced in \cite{pitman48} (as referenced in \cite{hajek99}). As noted in \cite{stu54a}, ``Pitman's result is equivalent to the use of estimating efficiency as a criterion''; some details on this will be given at the end of Section~\ref{sect:are}. 

It is well known (see e.g.\ \cite[page~844]{kruskal58}) that, for the independence test for \emph{normal} samples, the $\are$ of both $T$ and $S$ with respect (w.r.) to Pearson's sample correlation coefficient is $9/\pi^2[\approx0.912]$ and hence 
the $\are$ of $T$ w.r.\ to $S$ is $1$. 
In \cite{xu-etal13}, it was suggested based on plotting that, again in the normal model, the $\are$ of $T$ w.r.\ to $S$ is $>1$ for all nonzero values of the true (population) correlation coefficient. This, and more, was actually proved in \cite{mono}. 

%

Farlie \cite{far60} considered dependence models given by the formula  
\begin{equation}\label{eq:lin}
	F_\th=F_0+\th\De,
\end{equation}
where for each value of a real parameter $\th$, $F_\th$ is a continuous joint cumulative distribution function (cdf) of a pair $(X,Y)$ of real-valued random variables (r.v.'s) with marginal cdf's $G$ and $H$, not depending on $\th$; 
\begin{equation}\label{eq:F_0=Gh}
	F_0(x,y)=G(x)H(y), 
\end{equation}
so that $\th=0$ corresponds to the independence of $X$ and $Y$; 
\begin{equation}\label{eq:De=}
	\De(x,y)=G(x)H(y)A(G(x))B(H(y)), 
\end{equation}
and $A$ and $B$ are some functions; here and elsewhere in the paper, $x$ and $y$ stand for arbitrary real numbers, unless otherwise indicated.
Models of the form \eqref{eq:lin}--\eqref{eq:F_0=Gh} may be referred to as \emph{linear} models. 

For the particular class of linear models with $\De(x,y)$ as in \eqref{eq:De=} with additional mild regularity assumtions on $A$ and $B$, it follows from \cite{far60} that the $\are$ of $T$ w.r.\ to $S$ for the independence test is $1$. 
Because a linear model is a local approximation of a model smooth w.r.\ to the parameter, and in view of the mentioned result for normal samples, one may conjecture that the $\are$ of $T$ w.r.\ to $S$ will ``usually'' be $1$ for the independence test, at least for smooth enough models.  

In this paper, it will be shown that this conjecture is mainly true. More specifically, for arbitrary continuous models of dependence, we shall provide a condition that is necessary and sufficient for the $\are$ of $T$ w.r.\ to $S$ for the independence test to be $1$. This necessary and sufficient condition may actually be viewed as a conjunction of two conditions: (i) a certain degree of smoothness w.r.\ to the parameter $\th$ and (ii) a certain degree of nondegeneracy. 
  
To operate in full generality, we found it useful, and maybe even necessary, to present a more general and simple approach to Pitman's $\are$. 
The resulting notion of $\are$ is, at least in some aspects, more broadly applicable. In particular, we require no specific rate of convergence of the alternative value of $\th$ to the null value $\th_0$ -- in contrast with e.g.\ \cite{noe55},  
\cite[Assumptions E and E$'$, pages 57 and 59]{hoeff55}, \cite[page~271]{fraser57}, \cite[\S\S 25.5--25.6]{kenstu79}, where only alternatives of the form $\th=\th_0+k/n^r$ were considered; here, $k$ and $r>0$ are real constants and $n$ is the sample size. Also, our derivation of the expression for the $\are$ is actually simpler than known ones. 

The rest of the paper is organized as follows. 

In Section~\ref{sect:are}, we present the mentioned more general and simple approach to Pitman's $\are$. As in all other relevant sources on the notion of Pitman's $\are$, there we require a certain condition of uniform convergence to normality of the test statistics. 

In Section~\ref{setting}, this uniform convergence for $T$ and $S$ is stated in Proposition~\ref{prop:S}, 
%
with explicit expressions for the corresponding asymptotic means and variances. The main result of this section and of the entire paper is Theorem~\ref{th:}, which provides the mentioned necessary and sufficient condition for the $\are$ of $T$ w.r.\ to $S$ for the independence test to be $1$. 
It is also explained in Section~\ref{setting} that this necessary and sufficient condition may be viewed as a smoothness-and-nondegeneracy condition. 

In Section~\ref{suff}, based on Theorem~\ref{th:}, we state a few easy to use sufficient conditions for the $\are$ of $T$ w.r.\ to $S$ for the independence test to be $1$. 

In Section~\ref{plackett}, these sufficient conditions are illustrated on a few common  infinitely smooth models of dependence: the mentioned ``linear'' model, including Farlie's one; the  bivariate normal (BVN) model; Plackett's model; and Frank's model. It appears that similar applications can be made for 
most, if not all, of the 22 particular models of dependence (copulas) listed in \cite[Table~4.1]{nelsen06}, as well as for many other ones. 

In Section~\ref{MICD}, 
carefully selected/constructed specific models of dependence are considered to 
show effects of the presence or absence of the smoothness and/or nondegeneracy parts of the necessary and sufficient condition for the $\are$ of $T$ w.r.\ to $S$ for the independence test to be $1$.  

The necessary proofs are deferred to Section~\ref{proofs}.



\section{A general notion of the asymptotic relative efficiency}
\label{sect:are}

Let 
$\e,\e_1,\e_2,\dots$ be random elements of some measurable space, defined on another measurable space $(\Om,\A)$, and let $(\P_\th)_{\th\in\Th}$ be a parametric family 
of probability measures on $(\Om,\A)$ such that the random elements $\e,\e_1,\e_2,\dots$ are i.i.d.\ w.r.\ to each of the probability measures $\P_\th$ with $\th\in\Th$; here the parameter space $\Th$ is assumed to be an open interval on the real line $\R$. 
As usual, let $\E_\th$ and $\Var_\th$ denote the expectation and variance w.r.\  to the probability measure $\P_\th$. 

Suppose that for each $j=1,2$ and each natural $n$ we have a statistic $T_{j,n}$, which is a function of $\e_1,\dots,\e_n$, such that 
\begin{equation}\label{eq:BE}
	\P_\th(Z_{j,n}(\th)\le z)\underset{n\to\infty}\longrightarrow\Phi(z)
\end{equation}
uniformly in $z\in\R$ and $\th\in\Th_0$, where $\Phi$ is the standard normal cdf, 
\begin{equation}\label{eq:Z_n}
	Z_{j,n}(\th):=\frac{T_{j,n}-\mu_j(\th)}{\si_j(\th)/\sqrt n}, 
\end{equation}
$\Th_0$ is a neighborhood of a given point $\th_0\in\Th$, and $\mu_j$ and $\si_j$ are real-valued functions on $\Th_0$ such that 
\begin{equation}\label{eq:si}
\text{$\si_j$ is continuous and strictly positive,}	
\end{equation}
whereas 
\begin{equation}\label{eq:mu}
\text{$(\mu_j(\th)-\mu_j(\th_0))(\th-\th_0)>0$ for all $\th\in\Th_0\setminus\{\th_0\}$;} 	
\end{equation}
that is, $\mu_j(\th)-\mu_j(\th_0)$ is of the same sign as $\th-\th_0$ for all $\th\in\Th_0\setminus\{\th_0\}$. 

One may refer to $\mu_j(\th)$ and $\si_j(\th)$ respectively as the asymptotic mean and the asymptotic standard deviation of the statistic $T_{j,n}$. Conditions \eqref{eq:BE} and \eqref{eq:mu} justify the use of the ``asymptotic $z$-tests'' 
\begin{equation}\label{eq:tests}
	\de_{j,n}:=\ii{Z_{j,n}(\th_0)>z_\al}
\end{equation}
to test the null hypothesis $H_0\colon\th=\th_0$ versus the right-side alternative 
$H_0\colon\th>\th_0$, where $\ii\cdot$ denotes the indicator, $\al\in(0,1)$, and $z_\al:=\Phi^{-1}(1-\al)$. 
In view of \eqref{eq:BE}, the asymptotic size of these tests will be $\al$; that is, 
\begin{equation*}
\E_{\th_0}\de_{j,n}=\P_{\th_0}\big(Z_{j,n}(\th_0)>z_\al\big)\to\al 	
\end{equation*}
as $n\to\infty$; here and elsewhere in this paper, $j=1,2$. 

Suppose now that $n_1\in\N$, $n_2\in\N$, and $\tth\in\Th$ vary together so that $n_1\to\infty$, $n_2\to\infty$,  $\tth\downarrow\th_0$, and  
\begin{equation}\label{eq:to1-be}
	\E_\tth\de_{j,n_j}=\P_\tth\big(Z_{j,n_j}(\th_0)>z_\al\big)\to1-\be 
\end{equation}
for $j=1,2$, where 
$\be\in(0,1-\al)$ is given; that is, the alternative value $\tth>\th_0$ and the sample sizes $n_1$ and $n_2$ are such that the type II error probabilities for both tests $\de_{1,n_1}$ and $\de_{2,n_2}$ converge to the prescribed value $\be$. One may note that the total of the prescribed limit error probabilities of both types, I and II, is assumed here to be less than $1$: $\al+\be<1$. 

By \eqref{eq:Z_n}, 
\begin{equation*}
	Z_{j,n_j}(\th_0)=Z_{j,n_j}(\tth)\,\frac{\si_j(\tth)}{\si_j(\th_0)}
	+\frac{\mu_j(\tth)-\mu_j(\th_0)}{\si_j(\th_0)/\sqrt n_j}. 
\end{equation*}
Therefore, \eqref{eq:to1-be} can be rewritten as 
\begin{equation*}
	\P_\tth\big(Z_{j,n_j}(\tth)>z_{\al,\be}\big)\to1-\be, \quad\text{where}\quad
	z_{\al,\be}:=\Big(z_\al-\frac{\mu_j(\tth)-\mu_j(\th_0)}{\si_j(\th_0)/\sqrt n_j}\Big)
	\frac{\si_j(\th_0)}{\si_j(\tth)}.  
\end{equation*}
In view of the uniform convergence in \eqref{eq:BE}, it is now easy to see that $z_{\al,\be}\to z_{1-\be}$. Also, by the condition $\tth\downarrow\th_0$ and the continuity and positivity of $\si_j$, we have $\si_j(\tth)\to\si_j(\th_0)>0$. It follows that 
\begin{equation}\label{eq:sqrt n_j}
	\sqrt{n_j}\sim(z_\al-z_{1-\be})\Big/\frac{\mu_j(\tth)-\mu_j(\th_0)}{\si_j(\th_0)}
\end{equation}
and hence 
\begin{equation}\label{eq:re}
	\frac{n_2}{n_1}\sim\frac{\si^2_2(\th_0)}{\si^2_1(\th_0)}\,
	\frac{\big(\mu_1(\tth)-\mu_1(\th_0)\big)^2}{\big(\mu_2(\tth)-\mu_2(\th_0)\big)^2};
\end{equation}
as usual, we write $A\sim B$ for $A/B\to1$. 
Note that the latter displayed expression does not depend on $\al$ and $\be$, as long as $\al>0$, $\be>0$, and $\al+\be<1$. The condition $\al+\be<1$ is needed in order to have $z_\al-z_{1-\be}>0$ -- because otherwise \eqref{eq:sqrt n_j} cannot hold, in view of the conditions $\tth\downarrow\th_0$, \eqref{eq:mu}, and \eqref{eq:si}. 

The above consideration justifies the following notion of the (right-side) asymptotic relative efficiency ($\are$) of the sequence $T_1:=(T_{1,n})_{n=1}^\infty$ of statistics relative to the sequence $T_2:=(T_{2,n})_{n=1}^\infty$ at the point $\th_0\in\Th$: 
\begin{equation}\label{eq:are}
	\are_{T_1,T_2}(\th_0):=\frac{\si^2_2(\th_0)}{\si^2_1(\th_0)}\,
	\lim_{\th\downarrow\th_0}\frac{\big(\mu_1(\th)-\mu_1(\th_0)\big)^2}{\big(\mu_2(\th)-\mu_2(\th_0)\big)^2}
\end{equation}
if the latter limit exists (in $[0,\infty]$). 

More generally, one can define the upper and lower versions of the $\are$ by replacing $\lim$ in \eqref{eq:are} by $\limsup$ and $\liminf$, respectively. 

Similarly, one can define the left-side and two-side versions of the $\are$ by replacing $\th\downarrow\th_0$ in \eqref{eq:are} by $\th\uparrow\th_0$ and $\th\to\th_0$, respectively. 
In what follows, the two-side version of the $\are$ will be considered, unless otherwise indicated. 

These definitions of the $\are$ are more general than known ones, and their applicability is more easily verified. Indeed, the only conditions imposed here in the derivation of \eqref{eq:re} on the functions $\mu_j$ and $\si_j$ in \eqref{eq:BE}--\eqref{eq:Z_n} are \eqref{eq:si} and \eqref{eq:mu}. These latter conditions may be compared with conditions A, B, C, D (or D$'$) in \cite{noe55}, which involve derivatives of $\mu_j$ up to a certain order and a particular rate of $\tth$ approaching $\th_0$; as for our condition \eqref{eq:BE}, it is essentially the same as condition D in \cite{noe55}. 

On the other hand, if for $j=1,2$ and some natural $k$ the functions $\mu_j$ are indeed $k$ times differentiable at $\th_0$ and 
\begin{equation}\label{eq:noether}
	\mu_j'(\th_0)=\dots=\mu_j^{(k-1)}(\th_0)=0<\mu_j^{(k)}(\th_0) 
\end{equation}
(cf.\ condition A in \cite{noe55}), 
then it follows from \eqref{eq:are} that 
\begin{equation}\label{eq:are-diff}
	\are_{T_1,T_2}(\th_0)=\frac{\si^2_2(\th_0)}{\si^2_1(\th_0)}\,
	\frac{\mu_1^{(k)}(\th_0)^2}{\mu_2^{(k)}(\th_0)^2}. 
\end{equation}

More generally, by the mean value theorem, it follows from the two-side version of \eqref{eq:are} that 
\begin{equation}\label{eq:are'}
	\are_{T_1,T_2}(\th_0)=\frac{\si^2_2(\th_0)}{\si^2_1(\th_0)}\,
	\lim_{\th\to\th_0}\frac{\mu'_1(\th)^2}{\mu'_2(\th)^2}
\end{equation}
whenever this limit is well defined and exists, and similarly for the right- and left-side versions of the $\are$. 
One may note here that, in view of the continuity of $\si_j(\th)$ in $\th$, $\are_{T_1,T_2}(\th)$ is continuous in $\th$ at $\th=\th_0$
whenever the limit in \eqref{eq:are'} is well defined and exists.   

We shall see that, in one of the models to be considered later in this paper (namely, in what we shall refer to as the MICD AL model), the Noether-like condition \eqref{eq:noether} does not hold. In contrast, formula \eqref{eq:are'} appears to hold almost universally. 

One may also note that, in contrast with formula \eqref{eq:are-diff}, the definition \eqref{eq:are} is invariant w.r.\ to any re-parametrization of the form $\tau=h(\th)$, where $h$ is any strictly increasing continuous function and $\tau$ is the new parameter, replacing $\th$. 

For all these reasons, the general approach to the notion of $\are$ shown here may be useful elsewhere. 

To satisfy condition \eqref{eq:BE} of uniform convergence to normality, one can use Berry--Esseen (BE)-type bounds. The importance, and the lack at the time, of such bounds was emphasized by Kiefer \cite{kiefer68} in 1968. Since then, there has been significant progress in this direction. For instance, BE-type bounds of order $1/\sqrt n$ were obtained for $U$-statistics -- see e.g.\ \cite{kor94,chen07}; for the Student statistic -- see e.g.\ \cite{bent96,
student_els}; for maximum likelihood estimators -- see e.g.\ \cite{michel-pfanzagl71,MLE-BE}; and for general classes of nonlinear statistics -- see e.g.\  \cite{chen07,nonlinear-publ}. 

Definition \eqref{eq:are} of the $\are$, which is most suitable for comparisons of statistical tests,  goes back to Pitman. This ``testing'' notion of the $\are$ is closely related to, and one may even say is essentially the same as, the notion of the $\are$ in statistical estimation. Indeed, consider the following heuristics. If \eqref{eq:BE} holds, then one will usually have $\E_\th T_{j,n}\approx\mu_j(\th)$ and $\Var_\th T_{j,n}\approx\si_j(\th)^2/n$ (for large $n$). So, if $\mu_j(\th)$ is strictly increasing and smooth enough in $\th$, then for $\tilde T_{j,n}:=\mu_j^{-1}(T_{j,n})$ one will have $\E_\th \tilde T_{j,n}\approx\th$ (so that $\tilde T_{j,n}$ is an asymptotically unbiased estimator of $\th$) and 
\begin{equation*}
\Var_\th\tilde T_{j,n}\approx(\mu_j^{-1})'(\mu_j(\th))^2\frac{\si_j(\th)^2}n
=\frac1{\mu_j'(\th)^2}\,\frac{\si_j(\th)^2}n,	
\end{equation*}
by the delta method (see e.g.\ \cite{nonlinear-publ}). 
Thus, the ``estimation $\are$'' (see e.g.\ \cite{stuart-ord})  
\begin{equation*}
	\lim_{n\to\infty}\frac{\Var_\th\tilde T_{2,n}}{\Var_\th\tilde T_{1,n}}
\end{equation*}
is in agreement with the expressions \eqref{eq:are-diff} (for $k=1$), \eqref{eq:are'}, and \eqref{eq:are} for the ``testing'' $\are$.

\section{Kendall's and Spearman's correlation statistics and their asymptotic relative efficiency}
\label{setting}
From now on, the random elements $\e,\e_1,\e_2,\dots$, introduced in the beginning of Section~\ref{sect:are}, will be 
\begin{equation}\label{eq:Ps}
	\text{i.i.d.\ random points}\quad P:=(X,Y),\quad P_1:=(X_1,Y_1),\quad P_2:=(X_2,Y_2),\quad \dots
\end{equation} 
in $\R^2$, that is, pairs of real-valued r.v.'s. 
For each natural $n$, let 
\begin{equation*}
	\X_n:=(X_1,\dots,X_n)\quad\text{and}\quad\Y_n:=(Y_1,\dots,Y_n). 
\end{equation*}

To avoid complications resulting from ties between observations, we shall assume that for each $\th\in\Th$ the distribution of $(X,Y)$ is continuous, in the sense that the joint cdf $F_\th$ defined by the formula 
\begin{equation}\label{eq:F}
	F_\th(x,y):=\P_\th(X\le x,Y\le y)
\end{equation}
for real $x,y$ is continuous; this is equivalent to the continuity of the marginal cdf's $G_\th$ and $H_\th$ of $X$ and $Y$, defined by the conditions 
\begin{equation}\label{eq:G,H}
	G_\th(x):=\P_\th(X\le x)\quad\text{and}\quad H_\th(x):=\P_\th(Y\le y)
\end{equation}
for real $x,y$. Any equality between r.v.'s will be understood as taking place almost surely w.r.\  to the probability measures $\P_\th$ for all $\th\in\Th$. 

Kendall's difference sign correlation statistic is defined by the formula 
\begin{equation}\label{eq:T}
	T_n:=\binom n2^{-1}\sum_{1\le i<j\le n}\sign(X_i-X_j)(Y_i-Y_j).  
\end{equation}
Thus, $T_n$ is a U-statistic (of the form \eqref{eq:U_n}) and hence the mentioned BE-type bounds for \break 
$U$-statistics are directly applicable to $T_n$. 

For $n\ge3$, consider also Spearman's rank correlation statistic
\begin{equation}\label{eq:S}
	S_n:=\frac{\sum_{i=1}^n[R_i(\X_n)-(n+1)/2]\,[R_i(\Y_n)-(n+1)/2]}
	{\sqrt{\sum_{i=1}^n[R_i(\X_n)-(n+1)/2]^2} \sqrt{\sum_{i=1}^n[R_i(\Y_n)-(n+1)/2]^2}},
\end{equation}
which is obtained from Pearson's sample correlation coefficient by the replacing the $X_i$'s and $Y_i$'s with their respective ranks 
\begin{equation}\label{eq:R}
	R_i(\X_n)=1+\sum_{j=1}^n\ii{X_i>X_j}\quad\text{and}\quad R_i(\Y_n)=1+\sum_{k=1}^n\ii{Y_i>Y_k}.  
\end{equation}

Since each of the sets $\{R_1(\X_n),\dots,R_n(\X_n)\}$ and $\{R_1(\Y_n),\dots,
R_n(\Y_n)\}$ coincides with the set $\{1,\dots,n\}$, the expression for $S_n$ in \eqref{eq:S} can be simplified: 
\begin{equation*}
	S_n=\frac{12}{n(n^2-1)}\,\sum_{i=1}^n R_i(\X_n)R_i(\Y_n)-3\frac{n+1}{n-1}. 
\end{equation*}
In view of considerations in \cite[Section~9(e)]{hoeff48}  and \cite[Section~5c]{hoeff63},  
$S_n$ is also a U-statistic (of the form \eqref{eq:U_n}), but with a kernel $h$ depending on $n$. Therefore, it appears preferable here to approximate $S_n$ by a U-statistic with a kernel not depending on $n$ -- as will be done in the proof of Proposition~\ref{prop:S}.

\begin{proposition}\label{prop:S}
Condition \eqref{eq:BE}--\eqref{eq:Z_n} holds with $T_{1,n}=T_n$, $T_{2,n}=S_n$, 
\begin{align}
	\mu_1(\th)&=\mu_T(\th):=4\E_\th F_\th(X,Y)-1, \label{eq:mu_T} \\ 
\si^2_1(\th)&=\si^2_T(\th):=16\Var_\th\big(2F_\th(X,Y)-G_\th(X)-H_\th(Y)\big) \label{eq:si_T} \\ 
	\mu_2(\th)&=\mu_S(\th):=12\E_\th G_\th(X)H_\th(Y)-3, \label{eq:mu_S} \\ 
\si^2_2(\th)&=\si^2_S(\th)
:=144\Var_\th\Big( (1-G_\th(X))(1-H_\th(Y)) \notag \\ 
&\ \qquad\qquad\qquad\qquad+\int_\R F_\th(X,y)\dd H_\th(y)
+\int_\R  F_\th(x,Y)\dd G_\th(x)\Big) \label{eq:si_S} 
\end{align}
uniformly in $z\in\R$ and $\th\in\Th_0$, for any $\Th_0\subseteq\Th$ such that 
\begin{equation}\label{eq:inf si>0}
	\min_{j=1,2}\inf_{\th\in\Th_0}\si^2_j(\th)>0.
\end{equation} 
\end{proposition}

All necessary proofs are deferred to Section~\ref{proofs}. 

\medskip
\hrule
\medskip

Consider the association function $a_\th\colon\R^2\to\R$ between the r.v.'s $X$ and $Y$ defined by the formula 
\begin{equation}\label{eq:a}
	a_\th(x,y):=\P_\th(X\le x,Y\le y)-\P_\th(X\le x)\P_\th(Y\le y)=F_\th(x,y)-G_\th(x)H_\th(y)
\end{equation}
for real $x,y$; cf.\ \eqref{eq:F} and \eqref{eq:G,H}. 
Obviously, $a_\th=0$ if and only if $X$ and $Y$ are independent w.r.\  to the probability measure $\P_\th$. The cases $a_\th\ge0$ and $a_\th\le0$ may then be referred to as the cases of positive and negative dependence between $X$ and $Y$ w.r.\  to $\P_\th$. 

Note that the values of the statistics $T_n$ and $S_n$ do not change if $X_1,\dots,X_n,
Y_1,\dots,Y_n$ are replaced by $\psi_1(X_1),\dots,\psi_1(X_n),
\psi_2(Y_1),\dots,\psi_2(Y_n)$, respectively, where $\psi_1$ and $\psi_2$ are any strictly increasing functions. 
Moreover, with the already assumed convention that any equality between r.v.'s will be understood almost surely w.r.\  to $\P_\th$ for all $\th\in\Th$, we see that the statistics $T_n$ and $S_n$ do not change if $X_1,\dots,X_n,Y_1,\dots,Y_n$ are replaced by $G_\th(X_1),\dots,G_\th(X_n),\break 
H_\th(Y_1),\dots,H_\th(Y_n)$, respectively. This follows because (w.r.\  to $\P_\th$) the r.v.'s $G_\th(X_i)$ and $H_\th(Y_j)$ have the uniform distribution $\U[0,1]$ on the interval $[0,1]$, which is of course continuous. 

Therefore, we may, and indeed will, assume in the rest of the paper that the marginal cdf's $G_\th$ and $H_\th$ do not depend on $\th$: 
\begin{equation}\label{eq:const marg}
	G_\th=G\quad\text{and}\quad H_\th=H
\end{equation}
for some cdf's $G$ and $H$ and for all $\th\in\Th$. 
In particular, if each of the cdf's $G$ and $H$ coincides with the cdf of $\U[0,1]$, then $F_\th$ is a copula; a systematic treatment of copulas is given in \cite{nelsen06}. 

Further, let us assume that $0\in\Th$ and 
\begin{equation}\label{eq:indep}
	F_0(x,y)=G(x)H(y)
\end{equation}
for all real $x$ and $y$; that is, the hypothesis that $\th=0$ means that the r.v.'s $X$ and $Y$ are independent.

We are now ready to state the main result: 

\begin{theorem}\label{th:} Suppose that $T_{1,n}=T_n$, $T_{2,n}=S_n$. Suppose also that conditions \eqref{eq:si} and \eqref{eq:mu} hold for $\th_0=0$, some neighborhood $\Th_0$ of $0$, and $\mu_1(\th),\si_1^2(\th), 
\mu_2(\th),\si_2^2(\th)$ as in Proposition~\ref{prop:S}. 
Finally, suppose that conditions  \eqref{eq:const marg} and \eqref{eq:indep} hold 
as well. 
Then the following conditions are equivalent to one another: 
\emph{
\begin{enumerate}[(I)]
	\item $\are_{T,S}(0)=1$; 
	\item $
	\dfrac{\E_\th a_\th(X,Y)}{\E_0 a_\th(X,Y)}\underset{\th\to0}\longrightarrow1$; 
	\item $
	\dfrac{\int_{\R^2}(F_\th-F_0)\dd\,(F_\th-F_0)}{\int_{\R^2}(F_\th-F_0)\dd F_0}\underset{\th\to0}\longrightarrow0$; that is, 
	$\dfrac{\int_{\R^2}a_\th\dd\,a_\th}{\int_{\R^2}a_\th\dd F_0}\underset{\th\to0}\longrightarrow0$.  
\end{enumerate}
}
\end{theorem}

\begin{remark}\label{rem:mu incr}
Condition \eqref{eq:mu} will be used in the proof of Theorem~\ref{th:} only to obtain the implication (I)$\implies$(II) in the statement of Theorem~\ref{th:}; cf.\ \eqref{eq:mu incr used}. So, the implications (II)$\iff$(III)$\implies$(I) in the statement of Theorem~\ref{th:} will hold even without condition \eqref{eq:mu}. However, condition \eqref{eq:mu} is needed to justify the use of the ``asymptotic $z$-tests'' \eqref{eq:tests}. 
\end{remark}

\begin{remark}\label{rem:SND}
The numerator, $\int_{\R^2}(F_\th-F_0)\dd\,(F_\th-F_0)$, of the ratio in condition~(III) of Theorem~\ref{th:} is a quadratic form in $F_\th-F_0$, and so, this numerator is $O(\th^2)$ if $F_\th$ is smooth enough in $\th$ in an appropriate sense. On the other hand, the denominator, $\int_{\R^2}(F_\th-F_0)\dd F_0$, of that ratio is asymptotically proportional to $\th$ (as $\th\to0$) if e.g.\ the derivative of $\int_{\R^2}F_\th\dd F_0$ in $\th$ at $\th=0$ exists and is nonzero. 
Thus, condition (III) of Theorem~\ref{th:} will be ``typically'' satisfied. So, ``typically'' one will have $\are_{T,S}(0)=1$. 
It is also clear from this consideration that condition (III), and hence condition (I), of Theorem~\ref{th:} may fail to hold if the denominator $\int_{\R^2}(F_\th-F_0)\dd F_0$ is degenerate in the sense that the derivative of $\int_{\R^2}F_\th\dd F_0$ in $\th$ at $\th=0$ exists but equals $0$. 
Thus, we shall refer to condition (III) as the smoothness-and-nondegeneracy (SND) condition. 
This remark will be elaborated on in Section~\ref{suff} and 
illustrated in Section~\ref{models}. 
\end{remark}

\section{General conditions sufficient for SND}\label{suff}

In this section, we shall present general and usually easy to verify general conditions sufficient for SND and hence for $\are_{T,S}(0)=1$. Some of these results will be used in Section~\ref{models} concerning specific models of dependence. 

First here, let us formalize the part of Remark~\ref{rem:SND} concerning the nondegeneracy. 
The following proposition provides sufficient conditions for the denominator of the ratio in condition (III) in Theorem~\ref{th:} to approach $0$ \emph{no faster} than a nonzero constant multiple of $\th$ as $\th\to0$. 

\begin{proposition}\label{prop:ND}
If 
\begin{equation}\label{eq:der ne0}
	\Big(\frac{\dd}{\dd\th}\int_{\R^2}F_\th\dd F_0\Big)\Big|_{\th=0}\ne0 
\end{equation}
or 
\begin{equation}\label{eq:derr ne0}
	\int_{\R^2}\Big(\frac{\dd}{\dd\th}F_\th\Big|_{\th=0}\Big)\,\dd F_0\ne0,  
\end{equation}
then 
\begin{equation}\label{eq:ND}
	\liminf_{\th\to0}\Big|\frac1\th\,\int_{\R^2}(F_\th-F_0)\dd F_0\Big|>0. 
\end{equation} 
\end{proposition}

In the rest of this section, various conditions sufficient for the numerator of the ratio in condition (III) in Theorem~\ref{th:} to approach $0$ \emph{faster} than a nonzero constant multiple of $\th$ as $\th\to0$ will be given. 
Thus, using propositions in the rest of this section together with 
Proposition~\ref{prop:ND}, one will have various sufficient conditions for SND and hence for $\are_{T,S}(0)=1$. 

The following proposition is obvious. 

\begin{proposition}\label{prop:variation}
Consider the variation distance 
\begin{equation}\label{eq:d}
d(\th):=\|F_\th-F_0\|:=\int_{\R^2}|\dd\,(F_\th-F_0)|
\end{equation}
and the Kolmogorov distance 
\begin{equation*}
\rho(\th):=\sup_{(x,y)\in\R^2}\big|F_\th(x,y)-F_0(x,y)\big| 
\end{equation*}
between the cdf's $F_\th$ and $F_0$. 
Then   
\begin{equation}\label{eq:num}
	\frac1\th\,\int_{\R^2}(F_\th-F_0)\,d(F_\th-F_0)\underset{\th\to0}\longrightarrow0
\end{equation}
whenever
\begin{equation}\label{eq:d rho}
	\tfrac1\th\,d(\th)\rho(\th)\underset{\th\to0}\longrightarrow0. 
\end{equation}
So, one has SND and hence $\are_{T,S}(0)=1$ if conditions \eqref{eq:ND} and \eqref{eq:d rho} hold.
\end{proposition}

Since $\rho(\th)\le d(\th)$, we immediately have 

\begin{corollary}\label{cor:variation}
If $d(\th)^2=\|F_\th-F_0\|^2=o(|\th|)$ as $\th\to0$ and condition \eqref{eq:ND} holds, then $\are_{T,S}(0)=1$.
\end{corollary}

\begin{remark}\label{rem:scheffe}
Suppose that $\mu$ is a (nonnegative) measure on $\mathcal B(\R^2)$ and that $F_\th$ is absolutely continuous w.r.\ to $\mu$ with density $f_\th$ for all $\th\in\Theta_0$. 
By the Scheff\'e theorem \cite{scheffe}, one then has $\|F_\th-F_0\|\underset{\th\to0}\longrightarrow0$
whenever 
\begin{equation}\label{eq:f_th->}
f_\th\underset{\th\to0}\longrightarrow f_0	
\end{equation}
$\mu$-almost surely.
\end{remark}

A stronger version of condition \eqref{eq:f_th->} is sufficient for \eqref{eq:num}: 

\begin{proposition}\label{prop:density}
Suppose that $\mu$ is a finite (nonnegative) measure on $\mathcal B(\R^2)$ and that $F_\th$ is absolutely continuous w.r.\ to $\mu$ with density $f_\th$ for $\th\in\Theta_0$ such that 
\begin{equation}\label{eq:density}
\limsup_{\th\to0}\frac{\|f_\th-f_0\|_\infty}{|\th|}<\infty,
\end{equation}
where $\|\cdot\|_\infty$ is the norm on the Lebesgue space $L^\infty(\R^2,\mu)$.  
Then \eqref{eq:num} holds. So, $\are_{T,S}(0)=1$ if condition \eqref{eq:ND} holds as well. 
\end{proposition}

Immediately from Proposition~\ref{prop:density}, we deduce 

\begin{corollary}\label{cor:density}
Suppose that $\mu$ is a finite (nonnegative) measure on $\mathcal B(\R^2)$ and that $F_\th$ is absolutely continuous w.r.\ to $\mu$ with density $f_\th$ for $\th\in\Theta_0$ such that 
$\frac{\partial}{\partial\th}f_\th(x,y)$ is uniformly bounded over $\th\in\Th_0$ and $(x,y)\in\R^2$.  
Then \eqref{eq:num} holds. So, $\are_{T,S}(0)=1$ if condition \eqref{eq:ND} holds as well.
\end{corollary}

\begin{proposition}\label{prop:F_th.diff}
If $F_\th(x,y)$ is differentiable in $\th$ at $\th=0$ uniformly over $(x,y)\in\mathbb{R}^2$ and $\frac{\partial}{\partial\th}F_\th|_{\th=0}$ is continuous and bounded, then \eqref{eq:num} holds. So, $\are_{T,S}(0)=1$ if \eqref{eq:ND} holds as well.
\end{proposition}

Immediately from Proposition~\ref{prop:F_th.diff}, we deduce 

\begin{corollary}\label{cor:dF2}
If $\frac{\partial}{\partial\th}F_\th(x,y)$ and $\frac{\partial^2}{\partial\th^2}F_\th(x,y)$ are uniformly bounded over $\th\in\Th_0$ and $(x,y)\in\R^2$, then \eqref{eq:num} holds.  So, $\are_{T,S}(0)=1$ if \eqref{eq:ND} holds as well. 
\end{corollary}

\section{\texorpdfstring{$\are_{T,S}$}{are} for specific models of dependence}\label{models}

\subsection{
Some infinitely smooth models}\label{plackett} 

The simplest of such models is the {\bf ``linear'' model} given by formulas \eqref{eq:lin}--\eqref{eq:F_0=Gh}; Farlie's condition \eqref{eq:De=} need not be assumed here. It then follows immediately by Proposition~\ref{prop:variation} that for the ``linear'' model we have $\are_{T,S}(0)=1$ whenever $\int_{\R^2}\De\dd F_0\ne0$. In particular, this yields the result by Farlie \cite{far60}, mentioned in the introduction. 

Next, let us consider the {\bf bivariate normal (BVN) model}. 
That is, here it is assumed that the pair $(X,Y)$ is BVN with zero means, unit variances, and correlation coefficient $\th\in(-1,1)$. Letting now the measure $\mu$ in Corollary~\ref{cor:density} 
stand (say) for the standard BVN distribution (with zero means, unit variances, and correlation coefficient $0$), we see that 
\eqref{eq:num} holds. 
Next, using standard sufficient conditions for differentiation under the integral sign with respect to a parameter (see e.g.\ \cite[Theorem~2.27b]{folland}), for all real $x,y$ we have 
\begin{align*}
	\frac{\dd}{\dd\th}F_\th(x,y)\Big|_{\th=0}
	&=\int_{-\infty}^x \dd u\int_{-\infty}^y\dd v\; \frac{\dd}{\dd\th}g_\th(u,v)\Big|_{\th=0} \\ 
	&=\int_{-\infty}^x \dd u\int_{-\infty}^y\dd v\; uv\,\vpi(u)\vpi(v)
	=\vpi(x)\vpi(y)>0,
\end{align*}
where $g_\th$ is the joint pdf of the pair $(X,Y)$ 
and  
$\vpi$ is the standard normal pdf. So, condition \eqref{eq:derr ne0} holds. Thus, by Proposition~\ref{prop:ND} and Corollary~\ref{cor:density}, we confirm that $\are_{T,S}(0)=1$ for the BVN model. 

Consider now {\bf Plackett's model} \cite{plackett65
} given by the equation
\begin{equation}\label{eq:plackett}
	F_\th(x,y)-G(x) H(y)=\th\,  (G(x)-F_\th(x,y)) (H(y)-F_\th(x,y))
\end{equation}
for real $\th>-1$ and all real $x,y$. Plackett \cite{plackett65} showed that, for any given real $\th>-1$ and $x,y$,   equation \eqref{eq:plackett} has a single root $F_\th(x,y)$ satisfying the Fr\'echet \cite{frechet51} necessary conditions 
\begin{equation*}
	\max(0,G(x)+H(y)-1)\le F_\th(x,y)\le\min(G(x),H(y)). 
\end{equation*}
Equation \eqref{eq:plackett} is quadratic in $F_\th(x,y)$, and the mentioned single root is given by the formula 
\begin{equation}\label{eq:plackett'}
	F_\th=C_{P;\th}(G,H):=\frac{2(\theta +1) G  H}
	{1+\theta  (G+H)+\sqrt{1+2 \theta  (G+H-2 G H)+\theta ^2 (G-H)^2}}. 
\end{equation}
In particular, $\th=0$ corresponds to the independence, in accordance with \eqref{eq:indep}. 
It is clear that $F_\th$ is infinitely smooth in $\th$ in a neighborhood $\Th_0$ of $0$. Also,  
\begin{equation}\label{eq:ND,plack}
	\frac{\dd}{\dd\th}F_\th\Big|_{\th=0}(x,y)=(1-G(x))G(x)(1-H(y))H(y),
\end{equation}
so that the integral in \eqref{eq:derr ne0} is 
\begin{equation*}
	\int_\R(1-G)G\dd G\,\int_\R(1-H)H\dd H=\Big(\frac16\Big)^2\ne0. 
\end{equation*}
Thus, by Corollary~\ref{cor:dF2}, $\are_{T,S}(0)=1$ for Plackett's model. 

For each real $\th>-1$, the function $C_{P;\th}$ defined by \eqref{eq:plackett'} is a copula. A large number of examples of smooth one-parameter families of copulas can be found in \cite{nelsen06}, every one of them representing a model of dependence; in particular, see \cite[Table~4.1]{nelsen06}. It appears that for most, if not all, of those models one will have $\are_{T,S}(0)=1$. 

For instance, {\bf Frank's model} \cite{frank79} is given by the formula 
\begin{equation}\label{eq:frank}
	F_\th=C_{F;\th}(G,H):=-\frac1\th\,
	\ln \Big(1-\frac{(1-e^{-\th G}) (1-e^{-\th H})}{1-e^{-\th }}\Big) 
\end{equation}
for $\th\in\R\setminus\{0\}$, with $F_0=GH$, by continuity, so that \eqref{eq:indep} holds; cf.\ \cite[Table~4.1, page~116, formula (4.2.5)]{nelsen06}. 
In this case as well, $F_\th$ is infinitely smooth in $\th$, and 
\begin{equation*}
	\frac{\dd}{\dd\th}F_\th\Big|_{\th=0}(x,y)=\tfrac12\,(1-G(x))G(x)(1-H(y))H(y); 
\end{equation*}
cf.\ \eqref{eq:ND,plack}. 
Thus, by Corollary~\ref{cor:dF2}, $\are_{T,S}(0)=1$ for Frank's model as well.

\subsection{MICD models: \texorpdfstring{\underline{M}ixtures}{M} of 
\texorpdfstring{\underline{I}ndependence}{I}
and \texorpdfstring{\underline{C}omplete}{C}
\texorpdfstring{\underline{D}ependence}{D}
}\label{MICD}

Here we shall introduce four models of dependence: $\AS$, $\AL$, $\OS$, and $\OL$. In each of these models, the distribution of the random point $(X,Y)$ is a mixture of the uniform distribution on a finite union $A$ of (at most four) rectangles with sides parallel to the coordinate axes and the uniform distribution on the union $B$ of (at most two) segments of a straight line.
%
The sets $A=A^\mod_\th$ and $B=B^\mod_\th$ are subsets of the square $[-\tfrac12,\tfrac12]^2$ that depend on the choice $\mod\in\{\AS,\AL,\OS,\OL\}$ of an MICD model and on the value of the parameter $\th\in[-1,1]$, as follows: 
\begin{align*}
	A^\AS_\th&=\{(x,y)\in[-\tfrac12,\tfrac12]^2\colon |x|\le\tfrac{1-\th}2\text{ \underline{and} }|y|\le\tfrac{1-\th}2\}
	, \\ 
	A^\OS_\th&=\{(x,y)\in[-\tfrac12,\tfrac12]^2\colon
	|x|\le\tfrac{1-\th}2\text{ \underline{or} }|y|\le\tfrac{1-\th}2\}, \\ 
	A^\AL_\th&=\{(x,y)\in[-\tfrac12,\tfrac12]^2\colon
	|x|\ge\tfrac\th2\text{ \underline{and} }|y|\ge\tfrac\th2\}, \\ 
	A^\OL_\th&=\{(x,y)\in[-\tfrac12,\tfrac12]^2\colon
	|x|\ge\tfrac\th2\text{ \underline{or} }|y|\ge\tfrac\th2\}, \\ 
B^\AS_\th&=B^\OS_\th=\{(x,x)\colon\tfrac{1-\th}2\le|x|\le\tfrac12\}, \\ 
B^\AL_\th&=B^\OL_\th=\{(x,x)\colon|x|\le\tfrac\th2\}  
\end{align*}
for $\th\in[0,1]$; for $\th\in[-1,0)$ we only need to replace $(x,x)$ in the above definitions of the sets $B^\mod_\th$ by $(x,-x)$. 

The first letter ($\AA$ or $\O$) in the symbols $\AS$, $\AL$, $\OS$, and $\OL$ refers to the use of ``and'' or ``or'' in the above definitions of the sets $A=A^\mod_\th$, whereas the second letter ($\SSS$ or $\L$) in those symbols refers to the use of $\le$ (``smaller values") or $\ge$ (``larger values") in the those definitions. The support sets $A^\mod_\th\cup B^\mod_\th$ for $\mod\in\{\AS,\AL,\OS,\OL\}$ and $\th=2/10$ are shown in Fig.\ \ref{fig:MICDsupp}. 

\begin{figure}[h]
	\centering		\includegraphics[width=.8\textwidth]{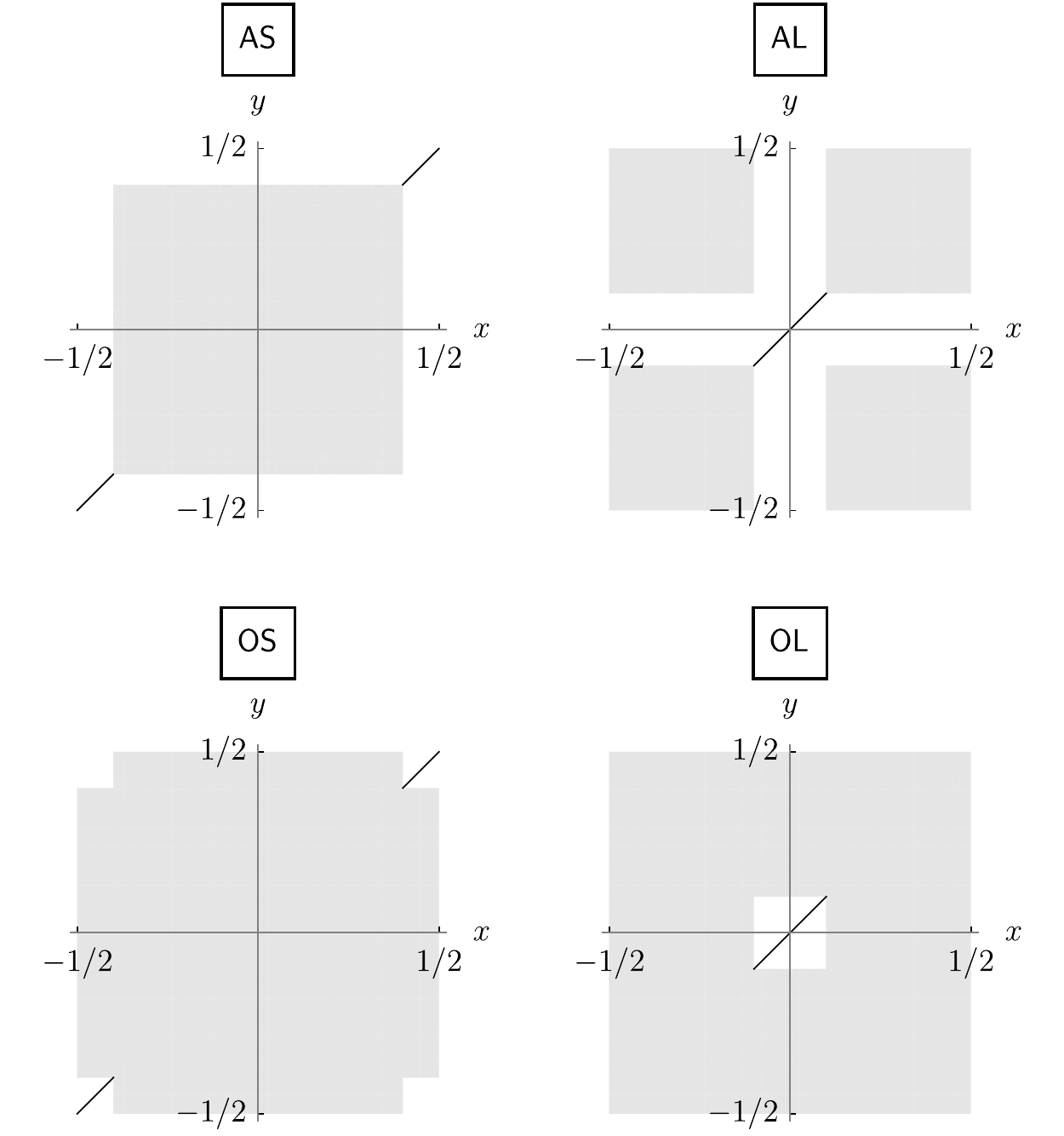}
	\caption{Support sets for the MICD models with $\th=2/10$.}
	\label{fig:MICDsupp}
\end{figure}

The definitions of the four MICD models are now completed by the condition that in each of these models the marginal distributions of $X$ and $Y$ be $\U[-\tfrac12,\tfrac12]$, that is, uniform on the interval $[-\tfrac12,\tfrac12]$. 

Thus, for each model $\mod\in\{\AS,\AL,\OS,\OL\}$ and each $\th\in[-1,1]$, the cdf $F^\mod_\th$ of the distribution of the random point $(X,Y)$ is a mixture of the cdf $F_{A^\mod_\th}$ of the uniform distribution on the set $A^\mod_\th$ and the cdf $F_{B^\mod_\th}$ of the uniform distribution on the set $B^\mod_\th$: 
\begin{equation}\label{eq:mix}
	F^\mod_\th=(1-w^\mod_\th)F_{A^\mod_\th}+w^\mod_\th F_{B^\mod_\th}, 
\end{equation}
where $w^\mod_\th:=|\th|$ if $\mod$ is one of the two ``and'' models, $\AS$ or $\AL$, whereas  $w^\mod_\th:=\th^2$ if $\mod$ is one of the two ``or'' models, $\OS$ or $\OL$. 
So, indeed in each of the MICD models the distribution of the random point $(X,Y)$ is a mixture of (at most four) distributions of random points with independent coordinates and the distribution of a random point on a straight line $y=x$ or $y=-x$, so that the coordinates of the latter random point are completely dependent of each other. 

Of course, in each of the MICD models the value $\th=0$ corresponds to the independence of $X$ and $Y$. 

One may note here that mixtures of uniform distributions over rectangles were briefly discussed, in similar contexts, on pages~821 and 834 of \cite{kruskal58}. Mixtures of distributions over (not necessarily straight) lines were considered on pages 829--832 there; see also e.g.\ \cite{schreyer-etal17} and references therein. 

The four MICD models may appear quite similar in spirit. However, as we will see, their $\are$ properties are very different. 

Analytically, the MICD models can be defined by the following formulas for the expected values of functions of the random point $(X,Y)$: 
\begin{align}
	\E^\AS_\th g(X,Y)&=\frac1{1-\th}\,\int_{[-(1-\th)/2,(1-\th)/2]^2}\dd u\,\dd v\, g(u,v) \notag \\ 
	&+\int_{[-1/2,-(1-\th)/2]\cup[(1-\th)/2,1/2]}\dd u\,g(u,u), \notag \\ 
	\E^\OS_\th g(X,Y)&=\int_{[-1/2,1/2]^2}\dd u\,\dd v\, g(u,v)\,\big(1-\ii{|u|\wedge|v|>(1-\th)/2}\big) \notag \\ 
	&+\th\,\int_{[-1/2,-(1-\th)/2]\cup[(1-\th)/2,1/2]}\dd u\,g(u,u), \notag \\ 
	\E^\AL_\th g(X,Y)&=\frac1{1-\th}\,\int_{[-1/2,1/2]^2}\dd u\,\dd v\, g(u,v)\,\ii{|u|\wedge|v|>\th/2} \notag \\ 
	&+\int_{[-\th/2,\th/2]}\dd u\,g(u,u), \notag \\ 
	\E^\OL_\th g(X,Y)&=\int_{[-1/2,1/2]^2}\dd u\,\dd v\, g(u,v)\,\big(1-\ii{|u|\vee|v|<\th/2}\big) \notag \\ 
	&+\th\,\int_{[-\th/2,\th/2]}\dd u\,g(u,u), \label{eq:E^OL} 
\end{align}
where $g$ is any bounded or nonnegative Borel-measurable function and 
$\E^\mod_\th$ denotes the expectation for model $\mod\in\{\AS,\AL,\OS,\OL\}$ and $\th\in[0,1]$; for $\th\in[-1,0)$ we need to replace $(u,v)$ and $(u,u)$ in these expressions by $(u,-v)$ and $(u,-u)$, respectively. 

Using the above formulas for $\E^\mod_\th g(X,Y)$ for 
$$g(u,v)=g_{x,y}(u,v):=\ii{u\le x,v\le y},$$ 
we can find expressions for the values $F^\mod_\th(x,y)=\E^\mod_\th g_{x,y}(X,Y)$ of the corresponding cdf's. 
Alternatively, one can obtain the joint cdf's $F^\mod_\th$ in the MICD models as certain linear combinations of the joint cdf's of uniform distributions on rectangles and on segments of straight lines. 

Recalling \eqref{eq:are'} and Proposition~\ref{prop:S}, and using \eqref{eq:E^OL} again, we find the following expressions for the asymptotic efficiency $\are_{T,S}^\mod(\th)$ of $T$ relative to $S$ for all the MICD models $\mod\in\{\AS,\AL,\OS,\OL\}$ and all $\th\in[0,1)$: 
\begin{align}
\are_{T,S}^\AS(\th)&=1, \label{eq:areAS} \\ 	
\are_{T,S}^\OS(\th)&=   	
\frac{9 (2 - 3 \th + 2 \th^2)^2 K_1(\th)}{10 (6 - 9 \th + 4 \th^2)^2 K_2(\th)}, \label{eq:areOS}  \\   
    K_1(\th)&:=10 + 10 \th + 10 \th^2 + 10 \th^3 - 230 \th^4 + 400 \th^5 - 267 \th^6 + 66 \th^7, 
    \notag \\  
   K_2(\th)&:=1 + \th + \th^2 + \th^3 - 11 \th^4 + 31 \th^5 - 27 \th^6 + 9 \th^7, \notag  \\ 
\are_{T,S}^\AL(\th)&=\frac{2 (1 + 2 \th + 3 \th^2 + 2 \th^3 + \th^4 + 9 \th^5)}{\th^2 (2 + 5 \th + 11 \th^2 + 
   18 \th^3)}, \label{eq:areAL}  \\ 	
\are_{T,S}^\OL(\th)&=   	
\frac{9 (5 + 4 \th^6 + 24 \th^7 - 33 \th^8)}{20 (1 + 8 \th^6 - 9 \th^8)},  \label{eq:areOL}    
\end{align}
with the convention that $\frac a0=\infty$ for $a>0$ concerning $\are_{T,S}^\AL$; 
these expressions for $\are_{T,S}^\mod(\th)$ have been computed with Mathematica. 
In particular, 
\begin{equation*}
	\are_{T,S}^\AS(0)=\are_{T,S}^\OS(0)=1,\quad\are_{T,S}^\AL(0)=\infty,\quad
	\are_{T,S}^\OL(0)=9/4.
\end{equation*}
For $\th\in(-1,0)$, one can use the symmetry $\are_{T,S}^\mod(\th)=\are_{T,S}^\mod(-\th)$ for all $\th\in(-1,1)$ and $\mod\in\{\AS,\AL,\OS,\OL\}$. 

Using \eqref{eq:areAS}--\eqref{eq:areOL}, it can be shown that $\are_{T,S}^\mod(\th)\ge1$ for all the MICD models $\mod\in\{\AS,\AL,\OS,\OL\}$ and all $\th\in[-1,1]$. Thus, at least for these models, Kendall's statistic is preferable to Spearman's in the $\are$ sense, not only for the independence tests, but also for testing for the dependence strength parameter $\th$. 

Graphs $\big\{\big(\th,\are_{T,S}^\mod(\th)\big)\colon\th\in[0,1)\big\}$ for $\mod\in\{\AS,\AL,\OS,\OL\}$ are shown in Fig.\ \ref{fig:MICDare}. 

\begin{figure}[h]
	\centering		\includegraphics[width=.9\textwidth]{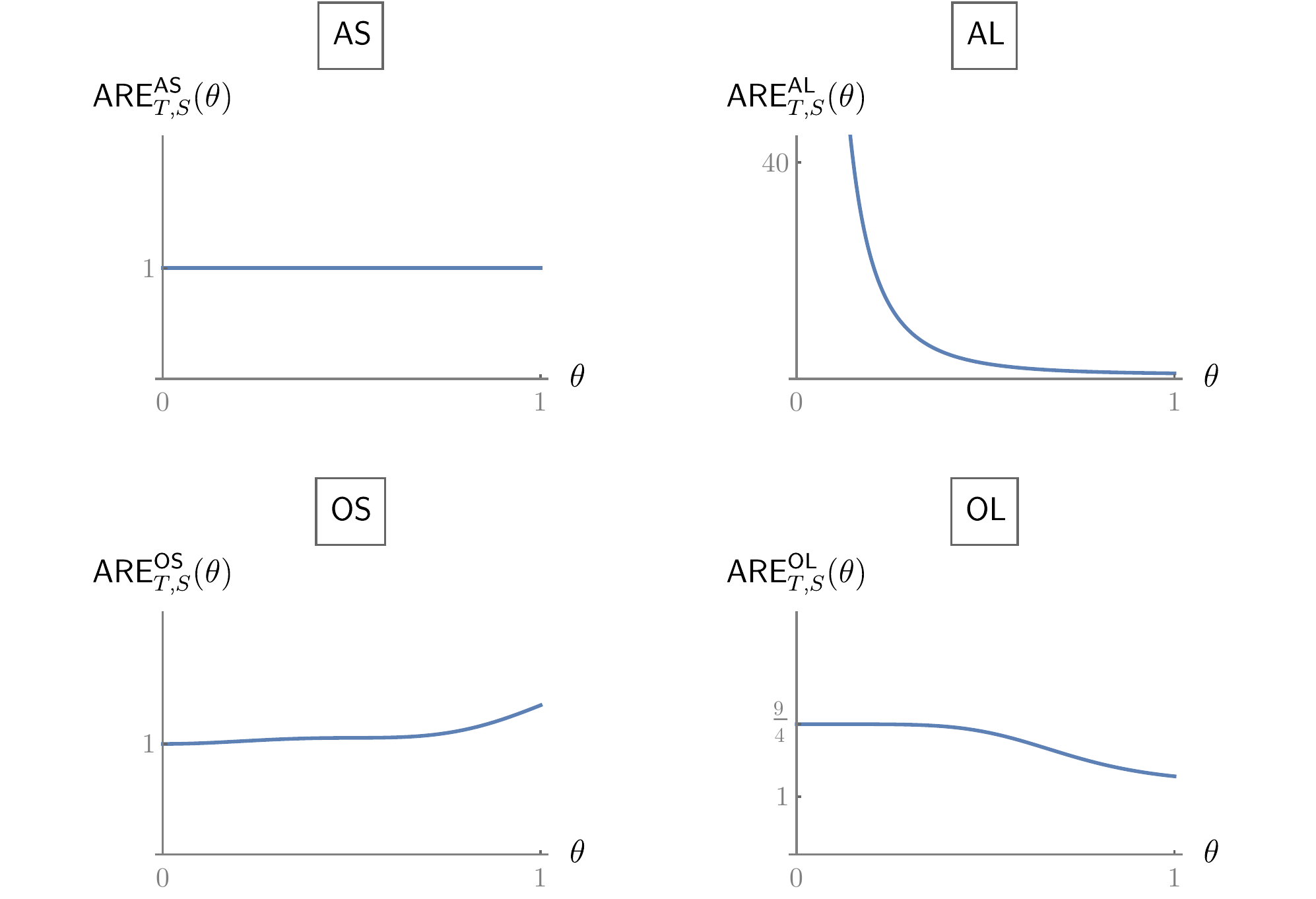}
	\caption{$\are_{T,S}^\mod$ for the MICD models.}
	\label{fig:MICDare}
\end{figure}

\bigskip
\hrule
\bigskip

In light of Theorem~\ref{th:} and results of Section~\ref{suff}, let us now 
explain the very different properties of the four MICD models concerning $\are_{T,S}$. 

It is easy to see that for all these models
\begin{equation}\label{eq:d=}
	d(\th)=O(w^\mod_\th), 
\end{equation}
with $d(\th)$ and $w^\mod_\th$ as in \eqref{eq:d} and \eqref{eq:mix}, respectively. 

If $\mod=\AS$, then $\frac{\dd}{\dd\th}F_\th(u,v)\Big|_{\th=0}=\frac14+uv$ for $u$ and $v$ in the interval $(-\frac12,\frac12)$, so that the integral in \eqref{eq:derr ne0} is $\frac14\ne0$. 
Also, by \eqref{eq:d=}, here $d(\th)=O(|\th|)$. 
So, by Corollary~\ref{cor:variation}, $\are^\AS_{T,S}(0)=1$ \big(in fact, as was already stated, $\are^\AS_{T,S}(\th)=1$ for $\th\in(-1,1)$\big).

If now $\mod=\OS$, then for all $u$ and $v$ in the interval $(-\frac12,\frac12)$ we have $\frac{\dd}{\dd\th}F_\th(u,v)\Big|_{\th=0}=0$ and $\frac{\dd^2}{\dd\th^2}F_\th(u,v)\Big|_{\th=0}=\frac12$. So, condition \eqref{eq:derr ne0} will hold only after the re-parametrization $\th\mapsto\th^2\sign\th$. This re-parametrization is indeed appropriate for the ``or'' models $\OS$ and $\OL$, with $w^\OS_\th=w^\OL_\th=\th^2$ and hence $d(\th)=O(\th^2)$ by \eqref{eq:d=}. So, by the ``re-parametrized'' version of Corollary~\ref{cor:variation}, we have $\are^\OS_{T,S}(0)=1$. 

Next, if $\mod=\AL$, then for all nonzero $u$ and $v$ in the interval $(-\frac12,\frac12)$ we  \break  
have $\frac{\dd}{\dd\th}F_\th(u,v)\big|_{\th=0}=(\frac12-|u|)(\frac12-|v|)\,\sign(uv)\ne0$, but the integral in the sufficient-for-nondegeneracy condition \eqref{eq:derr ne0} is $0$. In fact, the degeneracy here is rather deep, so that the denominator $\E_0 a_\th(X,Y)=\int_{\R^2}(F_\th-F_0)\dd F_0$ of the ratios in conditions (II) and (III) of Theorem~\ref{th:} is $\th^3/12$, which is much smaller for small $|\th|$ than the order of\;\ $\asymp|\th|$ that would be expected if the nondegeneracy condition \eqref{eq:ND} were satisfied. On the other hand, here $\E_\th a_\th(X,Y)=(3-\th) \th^2/12\asymp \th^2$, of the order ``normally'' expected in accordance with Remark~\ref{rem:SND} and the sufficient ``smoothness'' conditions stated in  Section~\ref{suff} after Proposition~\ref{prop:ND}. So, the mentioned deep degeneracy results in $\are_{T,S}^\AL(\th)$ going fast to $\infty$ as $\th\to0$, in accordance with \eqref{eq:areAL}. Therefore, by \eqref{eq:are'} and the continuity of $\si_j(\th)$ in $\th$, we end up with $\are_{T,S}^\AL(0)=\infty$. 

Finally, consider the case $\mod=\OL$. Recall that the degeneracy in the case $\mod=\AL$ occurs because for that model, whereas the derivative $\frac{\dd}{\dd\th}F_\th\Big|_{\th=0}$ is nonzero almost everywhere on the square $(-\frac12,\frac12)^2$, the average of this derivative given by the integral in \eqref{eq:derr ne0} is $0$. The degeneracy in the case $\mod=\OS$ was rather superficial, quickly cured by simple re-parametrization. In contrast, in the present case of $\mod=\OL$, the degeneracy occurs because 
$a_\th=F_\th-F_0$ is 
nonzero only on the square $(-\frac{|\th|}2,\frac{|\th|}2)^2$, quickly shrinking to the origin as $\th\to0$, and on this square $a_\th>0$. In fact, here 
\begin{equation}\label{eq:a^OL}
\begin{aligned}
	a_\th(u,v)&=F_\th(u,v)-F_0(u,v) \\ 
	&=
	\big(\tfrac{\th}2+\min(u,v)\big)\big(\tfrac{\th}2-\max(u,v)\big)
	\ii{\max(|u|,|v|)<\tfrac\th2}
\end{aligned}	
\end{equation}
for real $u,v$; 
here and in the rest of this consideration it is assumed that $\th\in(0,1)$; 
for $\th\in(-1,0)$, one needs to replace $\th$ and $v$ by $|\th|$ and $-v$, respectively. 
Therefore, 
the denominator of the ratios in conditions (II) and (III) of Theorem~\ref{th:} is 
\begin{equation}\label{eq:denOL}
	\int_{\R^2}a_\th\dd F_0=\int_{(-\frac{\th}2,\frac{\th}2)^2}a_\th(u,v)\dd u\dd v
	=\th^2\avee\limits_{(-\frac{\th}2,\frac{\th}2)^2} a_\th,
\end{equation}
where $\avee\limits_{(-\frac{\th}2,\frac{\th}2)^2} a_\th$ denotes the average value of $a_\th$ over the small square $(-\frac{\th}2,\frac{\th}2)^2$. 
On the other hand, in view of \eqref{eq:E^OL}, the numerator of the ratios in conditions (II) and (III) of Theorem~\ref{th:} is 
\begin{equation}\label{eq:numOL}
	\int_{\R^2}a_\th\dd F_\th=\th\,\int_{(-\frac{\th}2,\frac{\th}2)}a_\th(u,u)\dd u
	=\th^2\avee\limits_{\ell_\th} a_\th,
\end{equation}
where $\avee\limits_{\ell_\th} a_\th$ denotes the average value of $a_\th$ over the diagonal $\ell_\th$ of the small square $(-\frac{|\th|}2,\frac{|\th|}2)^2$ consisting of points of the form $(u,u)$. 
By \eqref{eq:a^OL}, $a_\th(u,v)<a_\th(u,u)$ for all points $(u,v)$ in the small square $(-\frac{|\th|}2,\frac{|\th|}2)^2$ that are off-diagonal, that is, with $v\ne u$. So, 
$\avee\limits_{(-\frac{\th}2,\frac{\th}2)^2} a_\th<\avee\limits_{\ell_\th} a_\th$ for $\th\in(0,1)$ and hence, by \eqref{eq:denOL} and \eqref{eq:numOL}, the ratio in condition (II) of Theorem~\ref{th:} is $>1$. Moreover, in view of \eqref{eq:a^OL}, $a_\th(u,v)$ is homogeneous of order $2$ in $\th$: 
\begin{equation*}
	a_\th(u,v)
	=\big(\tfrac{\th}2\big)^2\big(1+\min(\tu,\tv)\big)\big(1-\max(\tu,\tv)\big)
	\ii{\max(|\tu|,|\tv|)<1}
\end{equation*}
for real $u,v$, where $\tu:=u\big/\tfrac\th2$ and $\tv:=v\big/\tfrac\th2$. 
Therefore, the ratio in condition (II) of Theorem~\ref{th:} does not depend on $\th$ (and in fact equals $2$). Thus, $\are_{T,S}^\OL(0)\ne1$. 
This consideration is illustrated in Fig.\ \ref{fig:a^OL}, which shows the  
graph $\{(u,v,a_\th(u,v))\colon (u,v)\in(-\frac{\th}2,\frac{\th}2)^2\}$ for the MICD $\OL$ model with  $\th=\frac1{10}$. 
\begin{figure}[h]
	\centering		\includegraphics[width=.65\textwidth]{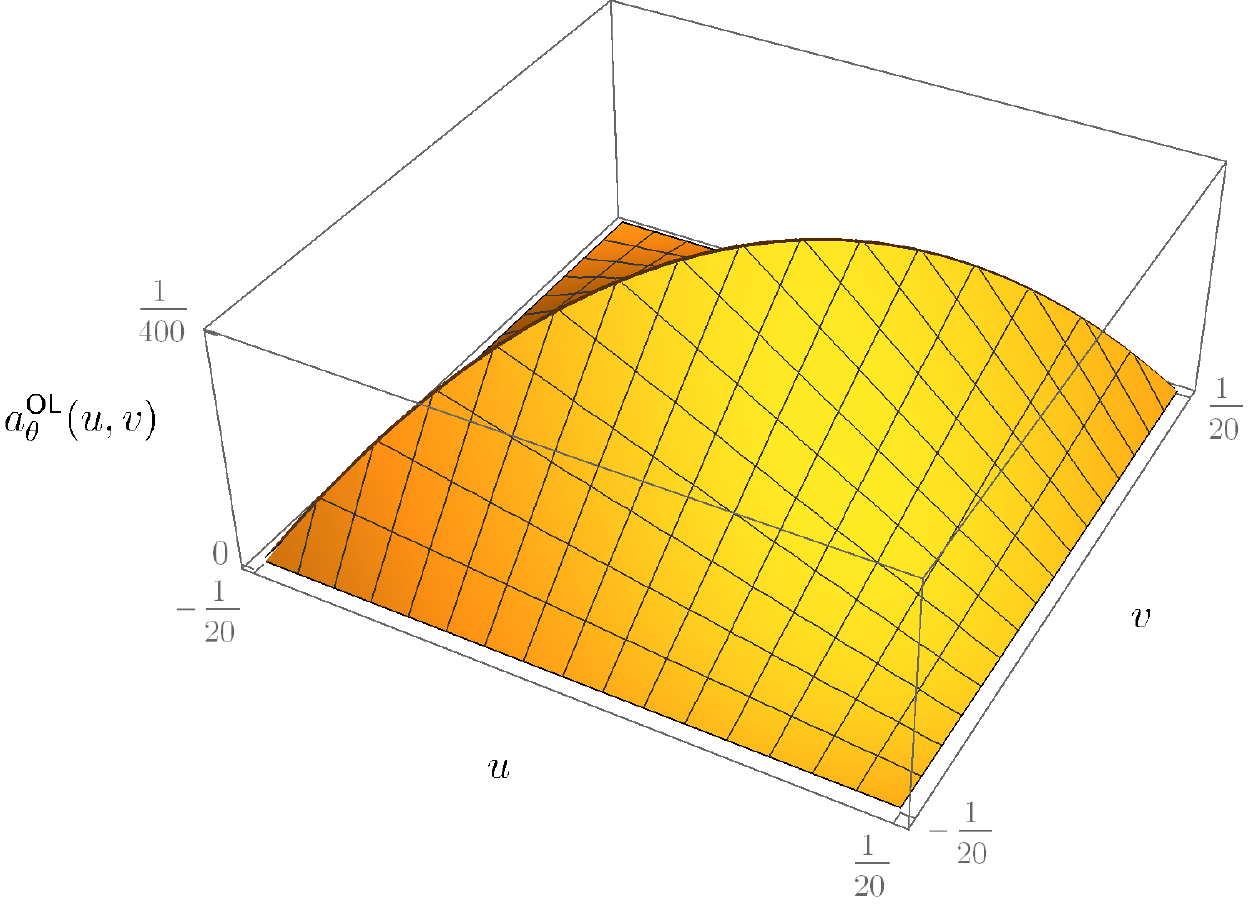}
	\caption{Graph of $a_\th$ over its support set $(-\frac{\th}2,\frac{\th}2)^2$ for the \emph{MICD} $\OL$ model with  $\th=\frac1{10}$.}
	\label{fig:a^OL}
\end{figure}

\section{Proofs}\label{proofs}

\begin{proof}[Proof of Proposition~\ref{prop:S}] 
This proof is based on a Berry--Esseen-type bound obtained by Chen and Shao \cite{chen07} for general U-statistics. 
The following lemma is a special case, sufficient for our purposes here, of formula (3.2) in \cite{chen07}.  

\begin{lemma}\label{lem:C-S}
Take any $\th\in\Th$. 
For $m\in\{2,3\}$ and natural $n\ge m$, let $h$ be a symmetric Borel-measurable function from $(\R^2)^m$ to $\R$ such that $|h|\le6$ and 
\begin{equation}\label{eq:si_h}
	\si_{h;\th}:=\sqrt{\Var_\th g_{h;\th}(P_1)}>0,\quad\text{where}\quad g_{h;\th}(P_1):=\E_\th\big(h(P_1,\dots,P_m)|P_1\big).
\end{equation} 
Let 
\begin{equation}\label{eq:U_n}
	U_n:=\binom nm^{-1}\,\sum_{1\le i_1<\cdots<i_m\le n}h\big(P_{i_1},\dots,P_{i_m}\big).  
\end{equation}
Then 
\begin{equation*}
	\sup_{z\in\R}\Big|\P_\th\Big(\frac{U_n-\E U_n}{m\si_{h;\th}/\sqrt n}\Big)-\Phi(z)\Big|\le\frac C{\si_{h;\th}^3\,\sqrt n}, 
\end{equation*}
where $C$ is a universal positive real constant. 
\end{lemma}

The condition in Lemma~\ref{lem:C-S} that the function $h$ is symmetric means, by definition, that $h$ is invariant w.r.\  to any permutation of its $m$ arguments. 

Introducing now the indicators 
\begin{equation}\label{eq:I,J}
	I_{ij}:=\ii{X_i>X_j}\quad\text{and}\quad J_{ij}:=\ii{Y_i>Y_j} 
\end{equation}
and recalling \eqref{eq:R}, we have 
\begin{align*}
	\sum_{i=1}^n R_i(\X_n)R_i(\Y_n)
	&=n+\sum_{i,j=1}^n I_{ij}+\sum_{i,k=1}^n J_{ik}+\sum_{i,j,k=1}^n I_{ij}J_{ik} \\ 
	&=n+\binom n2+\binom n2+\sum_{i,j,k=1}^n I_{ij}J_{ik} \\ 
	&=n^2+K_n+\binom n3\tS_n, 
\end{align*}
where 
\begin{equation*}
	K_n:=\sum_{i,j=1}^n I_{ij}J_{ij},   
\end{equation*}
\begin{equation}\label{eq:tS}
	\tS_n:=\binom n3^{-1}\,\sum_{1\le i<j<k\le n}h_3(P_i,P_j,P_k),  
\end{equation}
$P_i:=(X_i,Y_i)$, and 
\begin{equation}\label{eq:h_3}
	h_3(P_i,P_j,P_k):=I_{ij}J_{ik}+I_{ik}J_{ij}+I_{ji}J_{jk}+I_{jk}J_{ji}+I_{ki}J_{kj}+I_{kj}J_{ki}. 
\end{equation}
Here we have taken into account that $I_{ii}=J_{ii}=0$ for all $i$. 
So, in view of \eqref{eq:S}, 
\begin{equation*}
	S_n=2\frac{n-2}{n+1}\,\tS_n-3\frac{n+1}{n-1}+12\frac{n^2+K_n}{n(n^2-1)}. 
\end{equation*}
Since $0\le h_3(P_i,P_j,P_k)\le6$, 
$0\le\tS_n\le 6$, and $0\le K_n\le n^2$, we have 
\begin{equation}\label{eq:S_n=}
	S_n=2\tS_n-3+C_n/n;
\end{equation}
here and in what follows, $C_n$ will stand for various r.v.'s, possibly different even within the same expression, such that $|C_n|\le C$, where in turn $C$ denotes various universal positive real constants, also possibly different even within the same expression.

Let 
\begin{equation}\label{eq:mu_S,si_S}
	\mu_S(\th):=2\E_\th\tS_n-3,\quad 
	\si_S(\th):=6\si_{h_3;\th};  
\end{equation}
cf.\ \eqref{eq:si_h} and \eqref{eq:h_3}. 
Then, by \eqref{eq:S_n=}, \eqref{eq:tS}, and Lemma~\ref{lem:C-S}, for all real $z$ 
\begin{align*}
	\P_\th\Big(\frac{S_n-\mu_S(\th)}{\si_S(\th)/\sqrt n}\le z\Big)
	&=\P_\th\Big(\frac{\tS_n-\E_\th\tS_n}{3\si_{h_3;\th}/\sqrt n}\le z+\frac{C_n}{\si_{h_3;\th}\sqrt n}\Big) \\ 
	&\le\P_\th\Big(\frac{\tS_n-\E_\th\tS_n}{3\si_{h_3;\th}/\sqrt n}\le z+\frac{C}{\si_{h_3;\th}\sqrt n}\Big) \\ 	
	&\le\Phi\Big(z+\frac{C}{\si_{h_3;\th}\sqrt n}\Big)+ \frac{C}{\si_{h_3;\th}^3\sqrt n}\\ 	
	&\le\Phi(z)+ \frac{C}{\si_{h_3;\th}^3\sqrt n}; 	
\end{align*}
here at the last step we used the previously stated convention on the symbol $C$ and the fact that the standard normal density is bounded (by $1/\sqrt{2\pi}$) and hence $\Phi(w)-\Phi(z)\le C|w-z|$ for any real $z,w$; we also use used here the inequalities $0<\si_{h_3;\th}\le C$. 
Similarly,  
\begin{equation*}
	\P\Big(\frac{S_n-\mu_S(\th)}{\si_S(\th)/\sqrt n}\le z\Big)
\ge\Phi(z)- \frac{C}{\si_{h_3;\th}^3\sqrt n}. 
\end{equation*}
So, 
\begin{equation*}
	\sup_{z\in\R}\Big|\P\Big(\frac{S_n-\mu_S(\th)}{\si_S(\th)/\sqrt n}-\Phi(z)\Big| \le\frac{C}{\si_{h_3;\th}^3\sqrt n}\underset{n\to\infty}\longrightarrow0 	
\end{equation*}
uniformly over all $\th\in\Th_0$ -- for any $\Th_0\subseteq\Th$ such that 
\begin{equation*}
	\inf_{\th\in\Th_0}\si_{h_3;\th}>0. 
\end{equation*}

Thus, condition \eqref{eq:BE}--\eqref{eq:Z_n} holds with $T_{2,n}=S_n$, $\mu_2(\th)=\mu_S(\th)$, and	$\si_2(\th)=\si_S(\th)$, where $\mu_S(\th)$ and	$\si_S(\th)$ are as defined in \eqref{eq:mu_S,si_S}. 

Let us show that the definitions  
of $\mu_S(\th)$ and	$\si_S(\th)$ in \eqref{eq:mu_S,si_S} coincide with those in \eqref{eq:mu_S} and \eqref{eq:si_S}. 
In view of \eqref{eq:mu_S,si_S}, \eqref{eq:tS}, \eqref{eq:h_3}, and \eqref{eq:I,J}, 
\begin{multline*}
	\mu_S(\th)=2\E_\th\tS_n-3=2\E_\th h_3\big(P_1,P_2,P_3)-3
	=12\P_\th(X_1>X_2,Y_1>Y_3)-3
\end{multline*}
and 
\begin{align*}
	\P_\th(X_1>X_2,Y_1>Y_3)&=\int_{\R^2} \P_\th(x>X_2,y>Y_3)\P_\th(X_1\in\dd x,Y_1\in\dd y) \\
	&=\int_{\R^2} \P_\th(x>X_2)\P_\th(y>Y_3)\P_\th(X_1\in\dd x,Y_1\in\dd y) \\ 
	&=\E_\th G_\th(X_1)H_\th(Y_1),
\end{align*}
so that $\mu_S(\th)=12\E_\th G_\th(X)H_\th(Y)-3$; that is, 
the definition  
of $\mu_S(\th)$ in \eqref{eq:mu_S,si_S} coincides with that in \eqref{eq:mu_S}. 

Next, in view of \eqref{eq:si_h}, the independence and identical distribution of the pairs $P_i$'s, \eqref{eq:h_3}, \eqref{eq:I,J}, and \eqref{eq:Ps}, for any $(x,y)\in\R^2$ we have 
\begin{equation}\label{eq:g_h3}
	\begin{aligned}
	g_{h_3;\th}\big((x,y)\big)&=\E_\th h_3\big((x,y),P_2,P_3\big) \\ 
	&=\P_\th(x>X_2)\P_\th(y>Y_3)+\P_\th(x>X_3)\P_\th(y>Y_2) \\ 
&	+\P_\th(X_2>x,Y_2>Y_3)+\P_\th(X_2>X_3,Y_2>y) \\ 
&	+\P_\th(X_3>x,Y_3>Y_2)+\P_\th(X_3>X_2,Y_3>y) \\
&=2G_\th(x)H_\th(y)+2\P_\th(X>x,Y>Y_1)+2\P_\th(Y>y,X>X_1). 
\end{aligned}
\end{equation}
Next, 
\begin{equation}\label{eq:g_h3,1}
	\begin{aligned}
	\P_\th(X>x,Y>Y_1)&=\P_\th(X\le x,Y\le Y_1)-\P_\th(X\le x)-\P_\th(Y\le Y_1)+1 \\
	&=\E_\th F_\th(x,Y_1)-G_\th(x)+1/2  
\end{aligned}
\end{equation}
and similarly 
\begin{equation}\label{eq:g_h3,2}
	\P_\th(Y>y,X>X_1)=\E_\th F_\th(X_1,y)-H_\th(y)+1/2. 
\end{equation}
Combining \eqref{eq:g_h3}--\eqref{eq:g_h3,2}, we see that 
\begin{align}
	\tfrac12\,g_{h_3;\th}\big((x,y)\big)&=(1-G_\th(x))(1-H_\th(y))+\E_\th F_\th(x,Y)+\E_\th F_\th(X,y) \notag \\ 
	&=:\tilde g_{h_3;\th}\big((x,y)\big). \label{eq:tg} 
\end{align}
So, by \eqref{eq:mu_S,si_S} and \eqref{eq:si_h}, 
\begin{equation*}
	\si^2_S(\th)=144\Var_\th \tilde g_{h_3;\th}\big((X,Y)\big);  
\end{equation*}
that is, the definition  
of $\si_S(\th)$ in \eqref{eq:mu_S,si_S} coincides with that in \eqref{eq:si_S}. 


Thus, condition \eqref{eq:BE}--\eqref{eq:Z_n} holds with $T_{2,n}=S_n$, $\mu_2(\th)=\mu_S(\th)$, and	$\si_2(\th)=\si_S(\th)$, where $\mu_S(\th)$ and	$\si_S(\th)$ are now as defined in 
%
\eqref{eq:mu_S} and \eqref{eq:si_S}. 

To complete the proof of Proposition~\ref{prop:S}, it remains to verify 
that condition \eqref{eq:BE}--\eqref{eq:Z_n} holds with $T_{1,n}=T_n$, $\mu_1(\th)=\mu_T(\th)$, and	$\si_1(\th)=\si_T(\th)$, where $\mu_T(\th)$ and	$\si_T(\th)$ are as defined in \eqref{eq:mu_T} and \eqref{eq:si_T}. 
This verification is similar to, and significantly simpler than, the above verification for $T_{2,n}=S_n$. 
Here, in view of \eqref{eq:T} and \eqref{eq:I,J}, we have 
\begin{equation*}
	T_n=\binom n2^{-1}\sum_{1\le i<j\le n}h_2(P_i,P_j), 
\end{equation*}
where 
\begin{equation}\label{eq:h_2}
\begin{aligned}
	h_2(P_i,P_j)&:=
	\sign(X_i-X_j)(Y_i-Y_j) \\ 
	&=I_{ij}J_{ij}+(1-I_{ij})(1-J_{ij})-(1-I_{ij})J_{ij}-I_{ij}(1-J_{ij}) \\ 
	&=4I_{ij}J_{ij}-2I_{ij}-2J_{ij}+1. 
\end{aligned}	
\end{equation}
Then, quite similarly to the corresponding reasoning for $T_{2,n}=S_n$, we see that condition \eqref{eq:BE}--\eqref{eq:Z_n} holds with $T_{1,n}=T_n$ and 
\begin{equation}\label{eq:mu_T,si_T}
	\mu_T(\th):=\E_\th T_n=\E_\th h_2(P_1,P_2),\quad 
	\si_T(\th):=2\si_{h_2;\th}  
\end{equation}
uniformly over all $\th\in\Th_0$ -- for any $\Th_0\subseteq\Th$ such that 
\begin{equation*}
	\inf_{\th\in\Th_0}\si_{h_2;\th}>0, 
\end{equation*}
where $\si_{h_2;\th}$ is understood according to \eqref{eq:si_h}. 

So, to complete the proof of Proposition~\ref{prop:S}, it remains to verify that the definitions   
of $\mu_T(\th)$ and $\si_T(\th)$ in \eqref{eq:mu_T,si_T} coincide with those in \eqref{eq:mu_T} and \eqref{eq:si_T}. This is easy to do. Indeed, 
%
in view of \eqref{eq:si_h} and \eqref{eq:h_2}, for all $(x,y)\in\R^2$,
\begin{equation}\label{eq:g_h2}
	\begin{aligned}
	g_{h_2;\th}\big((x,y)\big)&=\E_\th h_2\big((x,y),P_2\big) \\ 
	&=4\P_\th(x>X_2,y>Y_2)-2\P_\th(x>X_2)-2\P_\th(y>Y_2)+1 \\ 
	&=4F_\th(x,y)-2G_\th(x)-2H_\th(y)+1. 
\end{aligned}
\end{equation}
So, in view of \eqref{eq:si_h}, the definition   
of $\si_T(\th)$ in \eqref{eq:mu_T,si_T} coincides with that in \eqref{eq:si_T}. 

Finally, 
by \eqref{eq:mu_T,si_T} and \eqref{eq:g_h2}, 
\begin{equation*}
	\begin{aligned}
	\mu_T(\th)=\E_\th h_2(P_1,P_2)&=\E_\th g_{h_2;\th}\big((X,Y)\big) \\ 
	&=4\E_\th F_\th(X,Y)-2\E_\th G_\th(X)-2\E_\th H_\th(Y)+1 \\ 
	&=4\E_\th F_\th(X,Y)-1, 
		\end{aligned}
\end{equation*}
so that the definition   
of $\mu_T(\th)$ in \eqref{eq:mu_T,si_T} coincides with that in \eqref{eq:mu_T}. 

Proposition~\ref{prop:S} is now completely proved. 
\end{proof}

\begin{proof}[Proof of Theorem~\ref{th:}] 
Let $U:=G(X)$ and $V:=H(Y)$. Then $U\sim\U[0,1]$ and $V\sim\U[0,1]$. Moreover, $U$ and $V$ are independent w.r.\ to $\P_0$. Therefore and in view of \eqref{eq:si_T} and \eqref{eq:si_S}, 
\begin{multline}\label{eq:si_T^2=}
	\si_T(0)^2=16\Var_0(2UV-U-V)=16\times4\,\Var_0(U-\tfrac12)(V-\tfrac12) \\ 
	=16\times4\,\E_0(U-\tfrac12)^2(V-\tfrac12)^2=16\times4\,(\tfrac1{12})^2
	=\tfrac49
\end{multline}
and 
\begin{equation}\label{eq:si_S^2=}
	\si_S(0)^2=144\Var_0\big((1-U)(1-V) +\tfrac U2+\tfrac V2\big)
	=144\Var_0(U-\tfrac12)(V-\tfrac12)
	=1. 
\end{equation}  
By \eqref{eq:mu_T} and \eqref{eq:mu_S},
\begin{equation*}
	\mu_T(\th)=4\E_\th F_\th(X,Y)-4\E_0 F_0(X,Y) 
\end{equation*}
and
\begin{equation}\label{eq:mu_S=}
	\mu_S(\th)=12\E_\th F_0(X,Y)-12\E_0 F_0(X,Y), 
\end{equation}
whence 
\begin{equation}\label{eq:3mu_T-mu_S}
	\tfrac1{12}\,\big(3\mu_T(\th)-\mu_S(\th)\big)=\E_\th (F_\th-F_0)(X,Y)=\E_\th a_\th(X,Y),  
\end{equation}
by \eqref{eq:a}. 
Here and in the rest of the proof, $\th\in\Th_0\setminus\{0\}$, so that \eqref{eq:mu} holds. 

Letting $(X_0,Y_0)$ and $(X_\th,Y_\th)$ be independent random points in $\R^2$ with cdf's $F_0$ and $F_\th$, respectively (w.r.\ to some probability measure $\P$ on $(\Om,\A)$),  
we have 
\begin{multline*}
	\E_\th F_0(X,Y)=\E F_0(X_\th,Y_\th)
	=\P(X_0<X_\th,Y_0<Y_\th) \\ 
	= 1-\P(X_\th<X_0)-\P(Y_\th<Y_0)+\P(X_\th<X_0,Y_\th<Y_0) \\ 
	=\P(X_\th<X_0,Y_\th<Y_0)
	=\E F_\th(X_0,Y_0)
	=\E_0 F_\th(X,Y),  
\end{multline*}
where $\E$ denotes the expectation w.r.\ to $\P$, and we used \eqref{eq:const marg} to observe that $\P(X_\th<X_0)=\frac12=\P(Y_\th<Y_0)$.
So, by \eqref{eq:mu_S=} and \eqref{eq:a}, 
\begin{equation*}
	\tfrac1{12}\,\mu_S(\th)=\E_0 F_\th(X,Y)-\E_0 F_0(X,Y)=\E_0 a_\th(X,Y). 
\end{equation*}
Comparing this with \eqref{eq:3mu_T-mu_S}, we see that 
\begin{equation}\label{eq:mu_T/mu_S}
	\frac{\mu_T(\th)}{\mu_S(\th)}=\frac13\,\Big(1+\frac{\E_\th a_\th(X,Y)}{\E_0 a_\th(X,Y)}\Big). 
\end{equation}
Note also that $\mu_T(0)=0=\mu_S(0)$. 
It now follows by the two-side version of \eqref{eq:are}, condition $\th_0=0$, \eqref{eq:si_T^2=}, \eqref{eq:si_S^2=}, \eqref{eq:mu}, and \eqref{eq:mu_T/mu_S} that 
\begin{equation}\label{eq:mu incr used}
	\are_{T,S}(0)=1\iff\lim_{\th\to0}\frac{\mu_T(\th)}{\mu_S(\th)}=\frac23
	\iff \lim\limits_{\th\to0}\dfrac{\E_\th a_\th(X,Y)}{\E_0 a_\th(X,Y)}=1. 
\end{equation}
So, conditions (I) and (II) in Theorem~\ref{th:} are equivalent to each other. 

The equivalence of conditions (II) and (III) in Theorem~\ref{th:} follows immediately, because 
\begin{align*}
	\E_0 a_\th(X,Y)
	=\int_{\R^2}F_\th\dd F_0-\int_{\R^2}F_0\dd F_0 =\int_{\R^2}(F_\th-F_0)\dd F_0 
\end{align*}
and
\begin{align*}
	\E_\th a_\th(X,Y)-\E_0 a_\th(X,Y)
	&=\int_{\R^2}(F_\th-F_0)\dd F_\th-\int_{\R^2}(F_\th-F_0)\dd F_0 \\ 
	&=\int_{\R^2}(F_\th-F_0)\dd\,(F_\th-F_0).
\end{align*}
The proof of Theorem~\ref{th:} is now complete. 
\end{proof}

\begin{proof}[Proof of Proposition~\ref{prop:ND}]
The implication \eqref{eq:der ne0}\,$\implies$\,\eqref{eq:ND} follows immediately from the definition of the derivative. The implication \eqref{eq:derr ne0}\,$\implies$\,\eqref{eq:ND} follows immediately by the Fatou lemma. 
\end{proof}

\begin{proof}[Proof of Proposition~\ref{prop:variation}]
This follows immediately because 
\begin{equation*}
\Big|\int_{\R^2}(F_\th-F_0)\,d(F_\th-F_0)\Big|
  \le\rho(\th)\int_{\R^2}|d(F_\th-F_0)|
  =2d(\th)\rho(\th). 
\end{equation*}
\end{proof}

\begin{proof}[Proof of Proposition~\ref{prop:density}]
By Remark~\ref{rem:scheffe}, condition \eqref{eq:density} implies $\|F_\th-F_0\|\underset{\th\to0}\longrightarrow0$. 
So, 
\begin{align*}
\Big|\frac1{\th}\int_{\R^2}(F_\th-F_0)\,\dd\,(F_\th-F_0)\Big|
&=\Big|\frac1{\th}\int_{\R^2}(F_\th-F_0)\,(f_\th-f_0)\,\dd\mu\Big| \\  &\le\|F_\th-F_0\|\,\frac{\|f_\th-f_0\|_\infty}{|\th|}\,\int_{\R^2}\dd\mu\underset{\th\to0}\longrightarrow0. 
\end{align*}
\end{proof}

\begin{proof}[Proof of Proposition~\ref{prop:F_th.diff}]
The uniform differentiability condition in this proposition means that
\begin{equation*}
s(\th):=\sup_{(x,y)\in\R^2}\bigg|\frac{F_\th(x,y)-F_0(x,y)}{\th}-\Big(\frac{\partial F_\th(x,y)}{\partial\th}\Big|_{\th=0}\Big)\bigg|\underset{\th\to0}\longrightarrow0.
\end{equation*}
So,
\begin{multline*}
\Big|\frac1{\th}\int_{\R^2}(F_\th-F_0)\,d(F_\th-F_0)\Big|  \\ 
  \le\bigg|\int_{\R^2}\Big(\frac{\partial F_\th}{\partial\th}\Big|_{\th=0}\Big)d(F_\th-F_0)\bigg| 
  +s(\th)\int_{\R^2}|d(F_\th-F_0)|\underset{\th\to0}\longrightarrow0,
\end{multline*}
since $s(\th)\rightarrow0$ implies $F_\th(x,y)\rightarrow F_0(x,y)$ for all $(x,y)\in\mathbb{R}^2$, whence $F_\th\Rightarrow F_0$ weakly.
\end{proof}

{\bf Acknowledgment.} On the author's request, R.\ Molzon assisted with a literature search, which resulted in additional references.


\def\cprime{$'$} \def\polhk#1{\setbox0=\hbox{#1}{\ooalign{\hidewidth
  \lower1.5ex\hbox{`}\hidewidth\crcr\unhbox0}}}
  \def\polhk#1{\setbox0=\hbox{#1}{\ooalign{\hidewidth
  \lower1.5ex\hbox{`}\hidewidth\crcr\unhbox0}}}
  \def\polhk#1{\setbox0=\hbox{#1}{\ooalign{\hidewidth
  \lower1.5ex\hbox{`}\hidewidth\crcr\unhbox0}}} \def\cprime{$'$}
  \def\polhk#1{\setbox0=\hbox{#1}{\ooalign{\hidewidth
  \lower1.5ex\hbox{`}\hidewidth\crcr\unhbox0}}} \def\cprime{$'$}
  \def\polhk#1{\setbox0=\hbox{#1}{\ooalign{\hidewidth
  \lower1.5ex\hbox{`}\hidewidth\crcr\unhbox0}}} \def\cprime{$'$}
  \def\cprime{$'$}

\end{document}